\documentclass[a4paper11pt,leqno]{article}
\usepackage{amsmath,ams symb}
\usepackage{amsfonts}
\usepackage{eucal}
\usepackage{amsthm}
\numberwithin{equation}{section}
\usepackage{comment}

\newcommand{\jap}[1]{\langle #1 \rangle}

\def\a{\alpha}
\def\b{\beta}

\def\d{\delta}
\def\e{\varepsilon}
\def\f{\varphi}
\def\g{\psi}

\def\k{\kappa}

\def\m{\mu}

\def\r{\rho}
\def\s{\sigma}
\def\t{\tau}
\def\x{\xi}
\def\y{\eta}
\def\z{\zeta}
\def\th{\theta}
\def\coloneqq{\mathrel{\mathop:}=}%

\def\re{\mathbb{R}}

\def\pa{\partial}

\renewcommand{\Im}{\text{{\rm Im}\;}}

\newcommand{\supp}{\text{{\rm supp}\;}}

\newtheorem{thm}{Theorem}[section]
\newtheorem{lem}[thm]{Lemma}
\newtheorem{prop}[thm]{Proposition}
\newtheorem{cor}[thm]{Corollary}

\theoremstyle{definition}

\newtheorem{ass}{Assumption}

\theoremstyle{remark}
\newtheorem{rem}[thm]{Remark}

\title{Strichartz estimates for non-degenerate Schr\"odinger equations}
\author{Kouichi Taira}%

\begin{document}
\maketitle

\begin{abstract}
We consider Schr\"odinger equation with a non-degenerate metric on the Euclidean space. We study local in time Strichartz estimates for the Schr\"odinger equation without loss of derivatives including the endpoint case. In contrast to the Riemannian metric case, we need the additional assumptions for the well-posedness of our Schr\"odinger equation and for proving Strichartz estimates without loss.
\end{abstract}

\section{Introduction}
We recall that a solution to the Schr\"odinger equation on the Euclidean space $\re^n$:
\begin{align}
i\pa_tu + \Delta u = 0,\quad u|_{t=0} = u_0 \in L^2(\re^n),
\label{eq:free}
\end{align}
satisfies following time in local Strichartz estimates:
\begin{align}
\|u\|_{L^p([-T,T], L^q(\re^n))}\leq C_{T}\|u_0\|_{L^2(\re^n)}, \quad \frac{2}{p}+\frac{n}{q}=\frac{n}{2},\quad (p,q,n)\neq (2,\infty,2).
\label{free St}
\end{align}
For recent thirty years, Strichartz estimates have been playing a important role in studying nonlinear Schr\"odinger equations and have been studied for the purpose of interesting in itself. By considering the Sobolev's embedding theorem, local in time Strichartz estimates suggests the solution to (\ref{eq:free}) gains the regularity compared to the initial value. These smoothing effects for Schr\"odinger equations are related to the dynamical properties of the associated geodesic (Hamiltonian) flow. Strichartz estimates first appear in \cite{SR}. Local in time Strichartz estimates on Euclidean spaces with non-trapping metrics are studied in \cite{BT1}, \cite{ST} and others.  On  compact manifolds, it is known that (\ref{free St}) holds if we replace $\|u_0\|_{L^2}$ by $\|u_0\|_{H^{\frac{1}{p}}}$ which is sharp for the standard spheres \cite{BGT}. On the other hand, on scattering manifolds with non-trapping geodesic flow (\ref{free St}) holds without loss (\cite{BT1}, \cite{M1}) and in recent years even time in global estimates has been proved (\cite{BT2}, \cite{MT}). We also know a little information for smoothing effects on the manifolds with the nonempty trapped set. For example, Bourgain \cite{B} proves $\e$-loss Strichartz estimates for $p=q=\frac{2(n+2)}{n}$ on the $1,2$ dimensional tori. Mizutani (\cite{M2}, \cite{M3}) proves the Strichartz estimates for the Schr\"odinger equations with the potentials which have growth at infinity and the loss depends on the growth rates of the potentials. A recent remarkable result is Burq-Guillarmou-Hassell's \cite{BGH}. They show Strichartz estimates without loss under the presence of the thin hyperbolic trapped set.

In recent years, non-elliptic Schr\"odinger equations have been attracted attention. Salort \cite{S} proves Strichartz estimates for Schr\"odinger equations associated with possibly degenerate metrics on $\re^n$ with $\frac{1}{p}+\e$ loss:
\begin{align*}
\|u\|_{L^p([-T,T], L^q(\re^n))}\leq C_{T}\|u_0\|_{H^{\frac{1}{p}+\e}(\re^n)}, \quad \frac{2}{p}+\frac{n+v}{q}=\frac{n+v}{2}, (p,q)\neq (2,\infty)
\end{align*}
where $v$ is a degenerate index of the metric. Mizutani and Tzvetkov \cite{MT} improve the result with $\frac{1}{p}$ loss  These results says the estimates are worse as $v$ increases and we can take $v=0$ for non-degenerate cases. In \cite{MT}, the same results are proved on compact manifolds. Wang \cite{W} shows that the loss of $\frac{1}{p}$ for $p=4$ is optimal on flat two torus by counting the number of the integer points surrounded by a hyperbola. Compared to Bougain's result \cite{B}, this shows the differences from Riemannian one's.  Another method for proving Strichartz estimates is used in \cite{M}, where they prove Strichartz estimates without loss under the compactly supported perturbation. Non-degenerate Schr\"odinger equations appears in older papers (\cite{D}, \cite{GS} ,\cite{KPRV1}, \cite{KPRV2}). They considere the well-posedness of linear and non-linear non-degenerate Schr\"odinger equations and the local smoothing effects. From another aspect, Chihara \cite{C} studies the well-posedness and the local smoothing effects of the dispersive equations which belong to the more general classes of the partial differential equations including the non-degenerate Schr\"odinger equations. 
 
We remark that for proving results in \cite{S} and \cite{MT}, we need the additional assumption, that is the energy estimates:
\begin{align*}
\|e^{itP}u_0\|_{H^s}\leq C_{T,s}\|u_0\|_{H^s}+C_{T,s}\int_0^t\|(i\pa_t+P)e^{irP}u_0\|_{H^s}dr,\quad 0\leq |t|\leq T
\end{align*}
hold for $s\geq 0$. The additional assumption comes from the fact that the Schr\"odinger operator $P$ and a elliptic operator do not commute in general. The energy estimates hold under non-trapping assumption on $\re^n$ for non-degenerate cases or under classical energy conservation assumption, which respectively correspond to Assumption \ref{assB} and Assumption \ref{assC}. The energy estimates imply $H^s$ well-posedness and a stability of a solution to (possibly non-elliptic) linear Schr\"odinger equaution. In this paper, we prove Strichartz estimates without loss for non-degenerate cases under the non-trapping conditions and the energy conservation conditions. More precisely, we prove Strichartz estimates both outside a compact set under the Assumption \ref{assC} and in any compact set under the Assumption \ref{assB} and \ref{assC}. This result is a generalization of the result in \cite{BT1}. 

Main difficulties in our results are the essential self-adjointness of $P$, the energy estimates (as explained above), estimates for the classical trajectories, and non-commutativity between the Paley-Littelewood cut off $\f(Q)$ and the propagator $e^{itP}$. To overcome the difficulty of the non-commutativity, we prove the commutator is very small (order $h$) by using Assumption \ref{assC} and Weyl calculus. By more precise calculating (compared to Riemannian cases \cite{BT1}), we can patch the frequency-localized estimates and we can conclude our results.

We only focus on the non-degenerate cases in this paper. We consider the Strichartz estimates for the solution to
\[
i\pa_tu + Pu = 0,\quad u|_{t=0} = u_0 \in L^2(\re^n)
\] 
where $P = -\pa_i g^{ij}(x)\pa_j$ denotes a divergence form operator with $g$ and $g(x)=(g_{ij}(x))$ is non-degenerate $n\times n$ matrix which has $k$ negative eigenvalues. Throughout this paper,  we use the Einsetin convention and omit the summation. A orbit $(z(t,x,\x), \z(t,x,\x))$ denotes the integral curve of $H_p$ with the initial data $(x,\x)\in T^*\re^n$ where $p=g^{ij}\x_i\x_j$ and $H_p$ denotes a Hamilton vector field with $p$. 

\begin{ass}[Long-range condition]
\label{assA}

There exists $0<\m<1$ such that for any $\a\in \mathbb{Z}_{\geq 0}^n$ there exists $C_{\a}>0$ such that,
\begin{eqnarray*}
\sum_{i,j}|\pa_x^{\a}(g^{ij}(x)-\d_{ij}^k)| \leq C_{\a}\jap{x}^{-\m - \a},\quad  \forall x\in\re^n.
\end{eqnarray*}
\end{ass}

\begin{ass}[Non-trapping condition]
\label{assB}

For any $(x,\x)\in T^*\re^n\setminus 0$, 
\[
|z(t,x,\x)| \to \infty,\quad \text{as}\quad |t|\to \infty
\]
\end{ass}

\begin{ass}[Positive energy conservation]
\label{assC}

There exists a real valued elliptic symbol $q\in S^2$ such that $\{p,q\}=0$.

\end{ass}

\begin{rem}
We can assume $q\geq C|\x|^2$ by adding a constant and $q^w(x,hD)\geq 0$ by the Fefferman-Phong inequality.

\end{rem}

Now, we state our main theorem.

\begin{thm}
Suppose Assumption \ref{assA} and Assumption \ref{assC}. Let $n\geq 2$ and $(p,q)$ satisfies the admissible conditon:
\begin{align}
p\geq 2,\quad q\geq 2,\quad \frac{2}{p}+\frac{n}{q}=\frac{n}{2},\quad (p,q,n)\neq (2,\infty,2).
\end{align}
Then, for $T>0$ there exist $R>0$ and $C>0$ such that 
\begin{align}
\|(1-\chi)e^{itP}u_0\|_{L^p([-T,T],L^q(\re^n))}\leq C\|u_0\|_{L^2(\re^n)}
\end{align}
 for $u_0\in L^2(\re^n)$ and $\chi \in C_c^{\infty}(\re^n)$ with $\chi = 1$ on $|x|\leq R$. In addition, if we suppose Assumption \ref{assB}, 
\begin{align}
\|e^{itP}u_0\|_{L^p([-T,T],L^q(\re^n))}\leq C\|u_0\|_{L^2(\re^n)}
\end{align}
holds for $u_0\in L^2(\re^n)$.

\end{thm}

\noindent{\bf Acknowledgment} The author would like to thank his supervisor Shu Nakamura for encourages writing this paper and helpful discussions. He also would be grateful to Hans Christianson for comments about local smoothing effects. This work was supported by JSPS Research Fellowship for Young Scientists, KAKENHI Grant Number 17J04478 and the program FMSP at the Graduate School of Mathematics Sciences, the University of Tokyo.

\section{Preliminary}

We first recall  the pseudo-differential operators. For any symbol $a\in C^{\infty}(\re^{2n})$, we define the pseudo-differential operator:
\[
a(x,D) u(x) =(2\pi)^{-n} \int e^{i(x-y)\cdot\x} a(x,\x) u(y) \,dy\,d\x,
\]
and the Weyl pseudo-differential operator:
\[
a^w(x,D) u(x) =(2\pi)^{-n} \int e^{i(x-y)\cdot\x} a((x+y)/2,\x) u(y) \,dy\,d\x.
\]
We also define symbol classes $S_m$, $S^l$ and $S^{l,m}$ by
\begin{align*}
&S_l \coloneqq \{a\in C^{\infty}(\re^{2n})\,|\, \forall\, \a, \b\;\exists\, C_{\a\b} \;\text{s.t.}\; |\pa_x^{\a}\pa_{\x}^{\b}a(x,\x)| \leq C_{\a\b}\jap{x}^{l-|\a|}\},\\
&S^m \coloneqq \{ a\in C^{\infty}(\re^{2n}) \,|\, \forall\, \a, \b\;\exists\, C_{\a\b} \;\text{s.t.}\; |\pa_x^{\a}\pa_{\x}^{\b}a(x,\x)| \leq C_{\a\b}\jap{\x}^{m-|\a|}\}\\
&S^{l,m} \coloneqq \{a\in C^{\infty}(\re^{2n})\,|\, \forall\,\a, \b\;\exists\, C_{\a\b} \;\text{s.t.}\;|\pa_x^{\a}\pa_{\x}^{\b}a(x,\x)| \leq C_{\a\b}\jap{x}^{l-|\a|} \jap{\x}^{m-|\b|}\},
\end{align*}
and $\Psi_l,\Psi^m,\Psi^{l,m}$ denote sets of the quantizations of $S_l, S^m, S^{l,m}$. Also,  $\Psi_{l,h}$, $\Psi_h^m,\Psi_h^{l,m}$ denote sets of the $h$-quantizations respectively.

\begin{lem}[Weyl full symbol of $P$]
\begin{align*}
p^w(x,D)=P-\sum_{i,j}\frac{1}{4}\pa_i\pa_jg^{ij}(x)
\end{align*}
\end{lem}
\begin{proof}
The proof is a easy computation.
\end{proof}

\begin{lem}
Let $a\in S^2$ be a real-valued elliptic symbol. Then, $a^w(x,hD)$ is essentially self-adjoint on $C_c^{\infty}(\re^n)$ and a domain of a self-adjoint extension of $a^w(x,hD)$ is $H^2(\re^n)$ for $0<h\leq 1$.
\end{lem}

\begin{proof}
It suffices for the essential self-adjointness to prove that $(a^w(x,hD)\pm i)u=0$ in a distribution sense implies $u=0$ for $u\in L^2(\re^n)$.  By the elliptic regularity, we know $u\in H^2(\re^n)$. Set $B=\jap{D}^{-1}a^w(x,hD)\jap{D}^{-1}\in \Psi^{0}$ and $v=\jap{D}u\in H^1(\re^n)$, then we have $Bv\pm i\jap{D}^{-2}v=0$. Notice that $B$ is the bounded self-adjoint operator on $L^2(\re^n)$. Then, $0=\Im (Bv\pm i \jap{D}^{-2}v,v)=\pm(\jap{D}^{-1}v,\jap{D}^{-1}v)=0$. This implies $u=\jap{D}^{-1}v=0$. To prove that the domain of a self-adjoint extension of $a^w(x,hD)$ is $H^2(\re^n)$, we recall this domain is $\{u\in L^2(\re^n)\,|\, a^w(x,hD)u\in L^2(\re^n)\}$ by a general argument in functional analysis. However, this set is exactly $H^2(\re^n)$.
\end{proof}

Thus, $\{q^w(x,hD)\}_{0<h\leq 1}$ extends to a family of self-adjoint operators with the common domains $H^2(\re^n)$. $q^w(x,hD)$ also denotes the self-adjoint extension.

\begin{cor}[Essential self-adjointness of $P$] 
Under Assumption A and Assumption C, $P$ is essentially self-adjoint on $C_c^{\infty}(\re^n)$ 
\end{cor}
\begin{proof}
As mentioned above, we may assume $q^w(x,D)\geq 1$ by the Fefferman-Phong inequality.
By the Nelson's commutator theorem (\cite{RS}) and $P\in \Psi^2$, it suffices to show that there exists $C>0$ such that $(u,[P,q^w(x,D)]u)\leq C(u,q^w(x,D)u)$ for $u\in C_c^{\infty}(\re^n)$. Since $\{p,q\}=0$ we know $[P,q^w(x,D)] \in\Psi^2$. This completes the proof.
\end{proof}

Finally, we state the local smoothing effects. By using non-trapping assumption and constructing an escape function, we can show the following lemma. For the details, see \cite{C}.

\begin{lem}[Local smoothing effects \cite{C}]
Suppose Assumption \ref{assA} and Assumption \ref{assB}. For any $T>0$ and $\chi\in C_c^{\infty}(\re^n)$ there exists $C>0$ such that
\[
\|\chi e^{itP}u_0\|_{L^2([-T,T],H^{\frac{1}{2}}(\re^n))} \leq C\|u_0\|_{L^2(\re^n)}
\]
for $u_0\in L^2(\re^n)$.
\end{lem}


\section{Functional calculus}
Set $Q(h)=q^w(x,hD)$.

\begin{thm}
For every $\f\in C_c^{\infty}(\re)$, $\f(Q(h))$ is a Weyl $h$-pseudodifferential operator of order $0$ (with respect to $h$-variable). More precisely, $\f(Q(h))\in \Psi_h^{0,0}\cap \Psi^{0,-\infty}$. Moreover, $\f(Q(h))$ has a following asymptotic expansion:
\begin{align*}
\f(Q(h))\equiv \sum_{j=0}^{\infty} h^jq_j^w(x,hD)\quad \text{modulo}\quad h^{\infty}\Psi_h^{0,-\infty}
\end{align*}
where $q_j\in S^{0,-\infty}$ and\,  $\supp q_j \subset \supp \f \circ q$. Furthermore, we can take $q_0(x,\x) = \f(p(x,\x))$ and $q_1=0$.

\end{thm}

\begin{proof}
See \cite{DS}. The fact that $q_1=0$ is due to Weyl calculus.

\end{proof}

\begin{cor}
For any $N\geq 1$, $1\leq q \leq r \leq \infty$, there exists $C_{Nqr} >0$ such that for any $0<h\leq 1$,

\begin{align}\label{eq:H^sPL}
\|\f(Q(h)) - \sum_{j=0}^{N}h^jq_j^w(x,hD) \|_{L^q\to L^r} \leq C_{Nqr}h^{N+1+d(\frac{1}{r}-\frac{1}{q})},
\end{align}

and for any $s\geq 0$ there exist $C_s>0$ such that for $N> 2s-1$ and $0<h\leq1$

\begin{align}\label{eq:L^pPL}
\|\f(Q(h)) - \sum_{j=0}^N h^jq_j^w(x,hD) \|_{H^{-s}\to H^s} \leq C_s h^{N+1-2s}.
\end{align}
\end{cor}

\begin{rem}
In \cite{BT1} and  \cite{BGT}, this corollary is proved by a direct calculation. However, we use the fact that $\f(Q(h))$ is exactly a pseudodifferential operator and we can get the optimal estimates in remainder terms. Remark that this improvements are not needed in our paper and never new result.
\end{rem}

\begin{cor}\label{PLsum}
Let $T>0$ and $\f\in C_c^{\infty}(\re)$ be a Liittlewood-Paley partition of unity, that is $1=\sum_{j=-\infty}^{\infty}\f(\frac{r}{2^j})$. Then for any $p\in[0,\infty], q\in [2,\infty)$ there exists $C>0$ such that 
\begin{align*}
&\|u\|_{L^p([-T,T],L^q(\re^n))}\\
&\hspace{1cm} \leq C\|u\|_{L^{\infty}([-T,T],L^2(\re^n))}+C(\sum_{j=0}^{\infty}\|\f(Q(\frac{1}{2^j}))u\|^2_{L^p([-T,T],L^q(\re^n))})^{\frac{1}{2}}
\end{align*}
for any $u\in \mathcal{S}'(\re\times \re^n)$ if the right side is finite.
\end{cor}
\begin{proof}
See \cite{BGT}.
\end{proof}

\begin{lem}[Energy estimates]\label{energy estimates}
Let $T>0$ and $s\in \re$. Then there exists $C>0$ such that
\begin{align*}
\|e^{itP}\|_{H^s\to H^s} \leq Ce^{T}
\end{align*}
for $0\leq |t| \leq T$.
\end{lem}

\begin{proof}
We may assume $s\in \re_{\geq 0}$ by the duality argument.

Set $v(t) = (I+q^w(x,D))^{s/2}(I+\e q^w(x,D))^{-s/2-1}e^{itP}u_0$, $u_0\in H^s$. 
Then, $v(t)$ satisfies
\begin{align*}
i\pa_tv(t) + Pv(t) =& [(I+q^w(x,D))^{s/2}(I+\e q^w(x,D))^{-s/2-1},P]\\
&\times (I+\e q^w(x,D))^{s/2+1} (I+q^w(x,D))^{-s/2}v(t).
\end{align*}
By the symbol calculus, a principal symbol of 
\begin{align*}
A_{\e} =& [(I+q^w(x,D))^{s/2}(I+\e q^w(x,D))^{-s/2-1},P]\\
&\times (I+ \e q^w(x,D))^{s/2+1}(I+q^w(x,D))^{-s/2}
\end{align*}
is given by 
\begin{align*}
\frac{1}{i}\{(1+q(x,\x))^{s/2}(1+\e q(x,\x)^{-s/2-1}),p\}(1+\e q(x,\x))^{s/2+1}(1+q(x,\x))^{-s/2},
\end{align*}
and this symbol belongs to a symbol class $S^{0,-\infty}$ because $\{q,p\} \in S^{0,-\infty}$. Moreover, $A_{\e}$ is bounded in $\e$ in $\Psi^{0,-\infty}$. By the Calderon-Vaillancourt theorem, $\|A_{\e}\|_{L^2\to L^2}$ is bounded in $\e$. Notice $v_{\e}(t) \in H^2 \subset D(P)$, then we have
\begin{align*}
\frac{d}{dt}\|v(t)|\|_{L^2}^2 =& (\frac{1}{i}(A_{\e}-P)v_{\e}, v_{\e})_{L^2}(t) +(v_{\e}, \frac{1}{i}(A_{\e}-P)v_{\e})_{L^2}(t)\\
\leq & C \|v_{\e}\|_{L^2}^2
\end{align*}
where $C>0$ is independent of $\e>0$ and $t$. By the Gronwall inequality,
\begin{align*}
\|v_{\e}(t)\|_{L^2} \leq e^{C|t|}\|v_{\e}(0)\|_{L^2}
\end{align*}
holds. Then, taking $\e \to 0$ we conclude
\begin{align*}
\|(I+q^w(x,D))^{s/2}e^{itP}u_0\|_{L^2} \leq e^{C|t|}\|(I+q^w(x,D))^{s/2}u_0\|_{L^2}.
\end{align*}

\end{proof}

\begin{lem}
For any $T>0$ and $s,m \in \re$, there exists $C>0$ such that
\begin{align}
\label{eq:commPL}
\|[\f(Q),P]\|_{H^{s+m}\to H^{s}} \leq Ch^{m+1},\\
\label{eq:commPR}
\|[e^{itP},\f(Q)]\|_{H^{s+m}\to H^s} \leq Ch^{m+1}
\end{align}
for $0<h\leq 1$ and $0\leq |t| \leq T$.
\end{lem}

\begin{proof}
First, we calculate $h^2P = p^w(x,hD)+h^21/4\sum_{i,j=1}^n\pa_i\pa_j g^{ij}(x)$. Since
\begin{align*}
[\f(Q),p^w(x,hD)] = [\sum_{j=0}^N h^j q_j^w(x,hD), p^w(x,hD)] + h^{N+1}\Psi^{0,-\infty}_h,
\end{align*}
for each $N \in \mathbb{N}_{\geq 0}$, we get 
\begin{align*}
[q_0^w(x,hD), p^w(x,hD)] &=\frac{h}{i}\{q_0,p\}^w(x,hD) + h^3\Psi^{0,-\infty} \in h^3\Psi^{0,-\infty}_h,\\
\text{$[$}q_1^w(x,hD), p^w(x,hD)]   &= 0,  \quad \text{since}  \quad q_1=0,\\
\text{$[$}\sum_{j=2}^N h^j q_j^w(x,hD), p^w(x,hD)] &\in h^{3}\Psi^{0,-\infty}_h.
\end{align*}
Take $N>0$ large enough, we get (\ref{eq:commPL}). By the Duhamel's formula,
\begin{align*}
[e^{itP},\f(Q)] = i\int_0^t e^{isP}[P,\f(Q)]e^{i(t-s)P}ds
\end{align*}
holds. Then, (\ref{eq:commPR}) can be proved by the first inequality and the energy estimates.
\end{proof}

\section{The Isozaki-kitada parametrix}

 In this section, we construct approximation propagators of $e^{itP}$ on the outgoing and incoming regions by the Isozaki-Kitada method. The method was used for constructing a wave operator with a long-range Schr\"odinger operator in \cite{IK}.
 The argument is slight modification of the method in \cite{BT1} and \cite{R}. We note that unlike elliptic cases, the arguments below success only when the angles between positions and momentums is a bit small. However, these are nothing to do for our estimates. As a application, we obtain the Strichatz estimates outside a compact set.
 
Let us explain some notations. For $J \Subset (0,\infty)$ an open interval, small enough $0<\s<1$, large enough $R>0$, we consider the outgoing and incoming regions:
\begin{align*}
\Gamma^{\pm}(R,J,\s) \coloneqq \{(x,\x) \in \re^{2n} | |x|>R,|\x|^2\in J, \pm \frac{g_0(x,\x)}{|x||\x|} >-\s \},
\end{align*}
where we set $g_0(\x,\x)=\x_1^2+...+\x_{n-k}^2-\x_{n-k+1}^2-...-\x_{n}^2$. Moreover, $I_k$ denotes a $n\times n$ diagonal matrix defined by $x\cdot I_ky=g_0(x,y)$ for $x,y\in \re^n$.
\subsection{Eikonal equations}

\begin{thm}\label{eikonal}
For any open interval $J\Subset \re$, there exist a constant $R>0$, $0<\s_0<1$ for $0<\s<\s_0$ there exists a family of real valued functions $S_{\pm} \in C^{\infty}(\Gamma^{\pm}(R,J,\s))$ such that
\begin{align}
g_x(\pa_xS_{\pm}(x,\x), \pa_xS_{\pm}(x,\x)) = g_0(\x,\x), \quad (x,\x) \in \Gamma^{\pm}(R,J,\s),
\label{eq:eikonal}
\end{align}
and 
\begin{align}\label{ph}
|\pa_x^{\a}\pa_{\x}^{\b}(S_{\pm}(x,\x)-x\cdot \x)|\leq C_{\a\b}\jap{x}^{1-\m}.
\end{align}
Moreover, there exist a constant $R' >0$ and a family of real-valued functions $\{S_{\pm, R}\}_{R\geq R'} \subset C^{\infty}(\re^{2n})$ such that
\begin{align}\label{eq:eikonal2}
g_x(\pa_xS_{\pm, R}(x,\x), \pa_xS_{\pm, R}(x,\x)) = g_0(\x,\x), \quad (x,\x) \in \Gamma^{\pm}(R,J,\s)
\end{align}
and there exists $C_{\a \b}>0$ such that for any $R\geq 2R'$ 
\begin{align}\label{ph2}
|\pa_x^{\a}\pa_{\x}^{\b}(S_{\pm, R}(x,\x) - x \cdot \x)|\leq C_{\a\b}\min(\jap{x}^{1-\m-|\a|},R^{1-\m-|\a|}).
\end{align}
\end{thm}

\begin{proof}

First, we will establish some estimates about the classical trajectories.

\begin{lem}
There exists $C>0$ such that
\[
C^{-1}|\x|^2\leq |\z(t,x,\x)|^2\leq C|\x|^2
\]
for $(x,\x)\in T^*\re\setminus 0$ and $t\in \re$.
\end{lem}

\begin{proof}
We compute
\begin{align*}
\frac{d}{dt}q(z(t,x,\x),\z(t,x,\x))= \{p,q\}(z(t,x,\x),\z(t,x,\x))=0.
\end{align*}
This calculation and the energy conservation law imply the lemma.
\end{proof}

\begin{lem}\label{cltj}
There exists $0<\s_0<1$, for any $0<\s<\s_0$ and $J \Subset (0,\infty)$ interval, $\exists R>0$, $\exists e_0>0$, and $c_1>0$ such that
\begin{align}
\label{Mooure}
|z(t,x,\x)|\geq& e_0(|x|+|t||\x|)\\
|\z(t,x,\x) - \x| \leq& c_1\jap{x}^{-\m}\\
|z(t,x,\x) - x - 2tI_k\x| \leq& c_1|t|\jap{x}^{-\m},
\end{align}
for $\pm t\geq 0$ and $(x,\x) \in \Gamma^{\pm}(R,J,\s)$.
\end{lem}

\begin{proof}
Fix $R_0>0$, $0<\s_0<1$ and $(x, \x)\in \Gamma^{\pm}(R_0,J, \s_0)$. Take $0< T_0=T_0^{\pm}(R_0,x,\x)$ as the first time when $|z(t,x,\x)|\leq  R_0$. 
\begin{align*}
\frac{d |z(t)|^2}{dt} &= 2z(t)\frac{dz(t)}{dt} = 4z(t)g(z(t))\z(t)\\
&\geq 4g_0(z(t),\z(t)) -CR_0^{-\m}|z(t)||\z(t)|,
\end{align*}

\begin{align*}
\frac{d^2 |z(t)|^2}{d^2t} =& 8(g(z(t))\z(t),g(z(t))\z(t)) + 8z^i(t)\pa_kg^{ij}(z(t))g^{kl}(z(t))\z^l(t)\z^j(t)\\
&- 4z^i(t)g^{ij}(z(t))\pa_jg^{kl}(z(t))\z^k(t)\z^l(t)\\
\geq& 8|\z(t)|^2(1 - O(R_0^{-\m}))\\
\geq& M|\x|^2(1- O(R_0^{-\m})).
\end{align*}


Take a $R_0>0$ large such that
\begin{align*}
\frac{d^2 |z(t)|^2}{d^2t} \geq \frac{M}{2}|\x|^2
\end{align*}

for $0\leq \pm t \leq T_0$. Then we have
\begin{align*}
|z(t)|^2 &\geq |x|^2 +4g_0(x,\x)t+O(R_0^{-\m})|x||\x|t  + \int_0^t\int_0^r \frac{d^2 |z(s)|^2}{d^2s} dsdr\\
&\geq |x|^2 \mp (4\s+O(R_0^{-\m}))|x||\x|t +M|\x|^2t^2\\
&\geq e_0^2(|x|+|t\x|)^2
\end{align*}
for $t\in[0,T_0]$ and $(x,\x) \in \Gamma^{\pm}(R_0,J,\s_0)$, if $R_0$ is large and $\s_0$ is small.  Define $R= 2R_0/e_0$. Let $(x,\x)\in \Gamma^{\pm}(R,J,\s)$ for $0<\s<\s_0$, then we can take $\pm T_0 = \pm \infty$. Indeed, if $\pm T_0 < \pm \infty$ $|z(T_0)|\geq 2R_0 > R_0$ implies that there exists $\e >0$  $|z(t)|\geq R_0$ on $|T_0| \leq |t| \leq |T_0| +\e$. This is a contradiction.


This proves $|z(t,x,\x)|\geq e_0(|x|+|t \x|)$ for $(x,\x) \in \Gamma^{\pm}(R,J,\s)$.

Next, we compute
\begin{align*}
|\z(t,x,\x)-\x| &\leq C|\int_0^{t} \jap{z(s)}^{-1-\m}ds|\\
&\leq C \jap{x}^{-\m},
\end{align*}
and
\begin{align*}
|z(t)- x-2tI_k\x| &= 2|\int_0^{t} g(z(s))\z(s)ds -tI_k\x|\\
&\leq C|t| \jap{x}^{-\m},
\end{align*}
for $t \geq 0$ and $(x,\x) \in \Gamma^{\pm}(R,J,\s)$.
\end{proof}

In the following, we fix a $0<\s<\s_0$. By using successive approximation arguments, we have the following estimates.

\begin{lem}\label{cltj2} 
There exists $R_1>0$ such that for any $R\geq R_1$,
\begin{align*}
|\pa_x^{\a}\pa_{\x}^{\b} (z(t)-x-2tI_k\x)(t,x,\x)| \leq& C_{\a \b}|t| \jap{x}^{-|\a |-\m},\\
|\pa_x^{\a}\pa_{\x}^{\b} (\z(t,x,\x)-\x) | \leq& C_{\a \b} \jap{x}^{-|\a |-\m},
\end{align*}
for $(x,\x)\in \Gamma^{\pm}(J,\s,R)$ and $\pm t\geq 0$.
\end{lem}

\begin{lem}\label{perstable}
There exists $\e_0>0$ such that for any $0<\e\leq \e_0$ there exists $R_2>0$ such that
\[
(z(t,x,\x),\z(t,x,\x))\in \Gamma^{\pm}(J_{\e},\s+\e,R_2)
\]
where $(x,\x)\in \Gamma^{\pm}(J,\s,\frac{1}{2}R_2)$, $\pm t\geq 0$, and $J_{\e}=J\pm \e$.

\end{lem}

\begin{proof}
The proof is same as in \cite{BT2}.
\end{proof}

Let us denote $\mathcal{K}^{\pm}(\x,\s)= \{\th \in \mathbb{S}^{n-1}\, |\, \pm g_0(\th, \x) > -\s |\x|\}$ and $\Gamma^{\pm}(J,\s)=\{(\th,\x)\in \mathbb{S}^{n-1}\times \re^n\, |\, |\x|^2\in J, \pm g_0(\th,\x)>-\s|\x| \}$ for $\s \in (-1,1)$, $\x\in \re^n$ and $J\Subset (0,\infty)$. We remark for $\e>$ enough small, there exist $\y^{\pm}_1,...,\y^{\pm}_{k_{\pm}} \in \re^n\setminus 0$ with $|\y_j^{\pm}|^2\in J$ for $1\leq j\leq k_{\pm}$ and $b>0$ such that
\begin{align*}
\Gamma^{\pm}(J,\s)&\subset \bigcup_{j=0}^{k_{\pm}} \mathcal{K}^{\pm}(\y_j^{\pm}, \s+\e)\times B_{b/2}(\y_j^{\pm}) \\
&\subset \bigcup_{j=0}^{k_{\pm}} \mathcal{K}^{\pm}(\y_j^{\pm}, \s+\e)\times B_{b}(\y_j^{\pm}) \subset \Gamma^{\pm}(J_{\e},\s+2\e).
\end{align*}
In fact, we notice
\begin{align*}
\overline{\Gamma^{\pm}(J,\s)} \subset \bigcup_{\substack{\x \in \re^n\setminus 0,\\b>0} } \mathcal{K}^{\pm}(\x, \s+\e)\times B_{b/2}(\x)
\end{align*}
and the result follows from the compactness of $\overline{\Gamma^{\pm}(J,\s)}$.
\begin{lem}
For small $\e>0$ with $0<\s+2\e<\s_0$ there exists $R_3>0$ such that for any $\pm t \geq 0$, $j=1,...,k_{\pm}$ and $x\in \re^n$ with $|x|>R_3$ and $\frac{x}{|x|}\in \mathcal{K}^{\pm}(\y_j^{\pm},\s+\e)$, the mapping $\x \mapsto \z(t,x,\x)$ is a diffeomorphism from near $\overline{B_{b}(\y_j^{\pm})}$ onto its range and
\[
B_{b/2}(\y_j^{\pm})\subset \{\z(t,x,\x)\,|\, \x\in B_{b}(\y_j^{\pm})\}.
\]
\end{lem}

\begin{proof}
By Lemma \ref{cltj} and the inverse function theorem, for large $R>0$ the mapping $\x \mapsto \z(t,x,\x)$ is a diffeomorphism from a neighborhood of $\overline{B_{3b/2}(\y_j^{\pm})}$ onto its range for $|x|>R$. Next, we prove
\[
B_{b/2}(\y_j^{\pm})\subset \{\z(t,x,\x)\,|\, \x\in B_{b}(\y_j^{\pm})\}
\]
by a connectivity argument.
Set $A=A_{\pm}=\{\pm t\geq 0\,|\, \overline{B_{b/2}(\y_j^{\pm})} \subset \{\z(t,x,\x)| \x\in B_{b}(\y_j^{\pm})\} \}$. Notice that $A\neq \emptyset$ since $0\in A$. It suffices to prove that $A$ is open and closed. Take $t_n\in A$ with $t_n \to t$ and $\z\in \overline{B_{b/2}(\y_j^{\pm})}$. There exist $\x^{(n)} \in B_{b}(\y_j^{\pm})$ such that $\z=\z(t_n,x,\x^{(n)})$. By the Bolzano-Weierstrass Theorem, we can assume $\x^{(n)}\to \x$ for some $\x\in \overline{B_{b}(\y_j^{\pm})}$.  We notice $\z=\z(t,x,\x)$ and
\[
|\y-\y_j^{\pm}|=b\quad \text{implies} \quad |\z(t,x,\y)-\y_j^{\pm}|\geq \frac{3}{4}b
\]
for large $R>0$ and $|x|>R$ by Lemma \ref{cltj}. Thus, we conclude $|\x-\y_j^{\pm}|<b$ and the closedness of $A$. Next, we prove the openness of $A$. Let $t\in A$. There exists $h_0>0$ such that for $0\leq h<h_0$, $\overline{B_{b/2}(\y_j^{\pm})}\cap  \{\z(t\pm h,x,\x)\,|\, \x\in \pa B_{b}(\y_j^{\pm})\} = \emptyset$ holds. For $s\in (t-h_0,t+h_0)$, we can take a continuous function $f_t: \re^n \to \re$ which is continuous in $t$ and satisfies $f_s(\z)= 0$ on $ \{\z(s,x,\x)\,|\, \x\in \pa B_{b}(\y_j^{\pm})\}$, $f_s(\z)>0$ in $ \{\z(s,x,\x)| \x\in  B_{b}(\y_j^{\pm})\}$ and $f_s(\z)<0$ in $\re^n\setminus \{\z(s,x,\x)\,|\, \x\in \overline{B_{b}(\y_j^{\pm})}\}$. In fact, we may define $f_s(\z)= -(|\x(s,x,\z)-\x_j|-b)$ near $\{\z(s,x,\x)\,|\, \x\in  \pa B_{b}(\y_j^{\pm})\}$. Notice  that $f_t(\z) >0 $ for $\z\in B_{b/2}(\y_j^{\pm})$ and $f_s(\z)\neq 0$ on $(s,\z)\in (t-h_0,t+h_0)\times \overline{B_{b/2}(\y_j^{\pm})}$. Thus we have $f_s(\z)>0$ for $(s,\z)\in (t-h_0,t+h_0)\times \overline{B_{b/2}(\y_j^{\pm})}$ which implies $\overline{B_{b/2}(\y_j^{\pm})} \subset  \{\z(s,x,\x)\,|\, \x\in B_{b}(\y_j^{\pm})\}$ for $s\in (t-h_0,t+h_0)$ and $A$ is open.

\end{proof}

We denote a restriction to $(x/|x|,\z)\in \mathcal{K}^{\pm}(\y_j^{\pm}, \s+\e)\times B_{b}(\x_j)$ of the inverse mapping $\x\mapsto \z(t,x,\x)$ by $\x_j^{\pm}(t,x,\z)$. By the usual Hamilton-Jacobi theory,  the solution to a following equation:
\begin{align}
\pa_t\f(t,x,\z)-p(x,\pa_x\f(t,x,\z))=0
\label{eq:Ha}
\end{align}
 is given by
\begin{align*}
\f_j^{\pm}(t,x,\z)=x\cdot \z + \int_0^tp(x,\x_j^{\pm}(s,x,\z))ds
\end{align*}
for $(x/|x|,\z)\in \mathcal{K}^{\pm}(\y_j^{\pm}, \s+\e)\times B_{b}(\y_j^{\pm})$. By the uniqueness of the solution to (\ref{eq:Ha}), we can patch the local solutions $\f_j^{\pm}(t,x,\z)$ and get a solution to (\ref{eq:Ha}) for $(x,\z)\in \Gamma^{\pm}(J,\s,R)$. This solution is denoted by $\f^{\pm}(t,x,\z)\in C^{\infty}(\re\times \Gamma^{\pm}(J,\s,R))$.

Now, we construct a phase functions $S_{\pm}(x,\z)$. Define
\[
F(t,x,\z)=\f^{\pm}(t,x,\z)-\f^{\pm}(t,2RI_k\z,\z)
\]
for $(x,\z)\in \Gamma^{\pm}(J,\s,R)$. We recall that $g_0$ is the matrix which $(i,j)$ coefficient's is $\delta_{ij}^k$. Let us denote $y(t,x,\z)=z(t,x,\x^{j,\pm}(t,x,\z))$ for $(x/|x|,\z)\in  \mathcal{K}^{\pm}(\y_j^{\pm}, \s+\e)\times B_{b}(\y_j^{\pm})$ where we omit the notation $j,\pm$. We calculate
\begin{align*}
\pa_tF(t,x,\z)=&p(y(t,x,\z),\z)-p(y(t,2RI_k\z,\z),\z)\\
=&(y(t,x,\z)-y(t,2RI_k\z,\z))\\
&\times \int_0^1(\pa_xp)(uy(t,x,\z)+(1-u)y(t,2RI_k\z,\z)),\z)du
\end{align*}
where we use the conservation law $p(y(t,x,\z),\z)=p(x,\x_j^{\pm}(t,x,\z))$. 
Here we need the following refined estimates for modified classical trajectories.

\begin{lem}\label{cltjref}
There exists $R>0$ such that for $1\leq j\leq k_{\pm}$, $(x/|x|,\x)\in  \mathcal{K}^{\pm}(\y_j^{\pm}, \s+\e)\times B_{b/2}(\y_j^{\pm})$, $|x|>R$, and $0\leq\pm s\leq \pm t$, we have
\begin{align*}
\z(s,x,\x(t,x,\z))-\z=&O(\jap{|s|+|x|}^{-\m})\\
z(s,x,\x(t,x,\z))-x-2tI_k\z=&O(\jap{|s|+|x|}^{1-\m})\\
\pa_x^{\a}\pa_{\x}^{\b}(\z(s,x,\x(t,x,\z))-\x)=&O(\jap{|s|+|x|}^{-\m-1-|\a|})\\
\pa_x^{\a}\pa_{\x}^{\b}(z(s,x,\x(t,x,\z))-x-2tI_k\z)=&O(\jap{|s|+|x|}^{1-\m-|\a|})
\end{align*}
where we omit $j,\pm$.
\end{lem}

\begin{proof}
The proof is similar to in Lemma \ref{cltj} and \ref{cltj2}.
\end{proof}

By Lemma \ref{cltjref}, we have
\begin{align*}
|\pa_x^{\a}&\pa_{\z}^{\b}(y(t,x,\z)-y(t,2RI_k\z,\z))|\\
&=|\pa_x^{\a}\pa_{\x}^{\b}\int_0^1(x-2RI_k\z)\cdot (\pa_xy)(t,ux+(1-u)2RI_k\z,\z)du|\\
&\leq C_{\a\b}\jap{x}^{1-|\a|}
\end{align*}
and 
\begin{align*}
|uy&(t,x,\z)+(1-u)y(t,2RI_k\z,\z)|\\
\geq&|2tI_k\z+ux+(1-u)RI_k\z|-O(u(|t|+|x|)^{1-\m})\\
&-O((1-u)(|t|+2R|\z|)^{1-\m})\\
\geq& C(1+u|x|+|t|)
\end{align*}
uniformly in $\pm t\geq  0$ and $(x/|x|,\z)\in  \mathcal{K}^{\pm}(\y_j^{\pm}, \s+\e)\times B_{b}(\y_j^{\pm})$.

Notice that 
\begin{align*}
\pa_t\pa_x^{\a}\pa_{\z}^{\b}F(t,x,\z)=&\sum_{\substack{\a_1+\a_2=\a\\ \b_1+\b_2=\b}}C_{\a_1,\b_1}\pa_x^{\a_1}{\pa_\z}^{\b_1}((y(t,x,\z)-y(t,2RI_k\z,\z)))\\
&\cdot \int_0^1\pa_x^{\a_2}{\pa_\z}^{\b_2}((\pa_xp)(uy(t,x,\z)+(1-u)y(t,2RI_k\z,\z)),\z))du
\end{align*}

Using this calculation and Lemma \ref{cltjref}, we get
\[
\pa_t\pa_x^{\a}\pa_{\x}^{\b}F(t,x,\z)=O(|t|^{-1-\m})
\]
uniformly in any compact subset in $(|x|,x/|x|,\z)\in (R,\infty)\times \mathcal{K}^{\pm}(\y_j^{\pm}, \s+\e)\times B_{b}(\y_j^{\pm})$. This implies
\[
S_{\pm}(x,\z)=x\cdot \z + \int_0^{\infty}\pa_tF(t,x,\z)dt
\]
exists and $S$ is smooth in $(x,\z)\in \Gamma^{\pm}(J,\s,R)$. We notice
\begin{align}
&\lim_{t\to \pm \infty}|\pa_{\z}\f^{\pm}(t,x,\z)|=\infty,\\
&\pa_xS_{\pm}(x,\z)=\lim_{t\to \pm \infty}\pa_{x}\f^{\pm}(t,x,\z)
\end{align}
hold for $(x,\z)\in \Gamma^{\pm}(J,\s,R)$. These relations and the energy conservation law imply
\begin{align}
p(x,\pa_xS_{\pm}(x,\z))=g_0(\z,\z)
\end{align}
for  $(x,\z)\in \Gamma^{\pm}(J,\s,R)$. This proves (\ref{eq:eikonal}). Next, we prove (\ref{ph}). For $|\a|=|\b|=0$, we comupute
\begin{align*}
|\int_0^1\int_0^{\pm}\pa_t\pa_x^{\a}\pa_{\z}^{\b}F(t,x,\z)dtdu|\leq& C\jap{x}^{1-|\a|}|\int_0^1\int_0^{\pm\infty}(1+u|x|+|t|)^{-1-\m} dtdu|\\
\leq&C\jap{x}^{1-\m-|\a|} 
\end{align*}
and we get $|\pa_x^{\a}\pa_{\z}^{\b}(S_{\pm}(x,\z)-x\cdot \x)|\leq C\jap{x}^{1-\m-|\a|}$ for $(x,\x)\in \Gamma^{\pm}(J,\s,R)$.

For the proof of (\ref{eq:eikonal2}) and (\ref{ph2}), we refer to \cite{BT1}.
\end{proof}




\begin{prop}[Approximate propagators]\label{pr:parx}
Let us consider $J \Subset J_1 \Subset J_2 \Subset (0,\infty)$ and small numbers $0\leq \s <\s_1<\s_2 <1$. Then, there exists $R_0 > 1$ such that for any $N \in \mathbb{N}_{>0}$, $k\in \mathbb{N}_{\geq 0}$, $R\geq R_0$ and $\chi_{\pm} \in S^{-k,-\infty}$ supported in $\Gamma^{\pm}(R,J,\s)$ we can take $a_j^{\pm} \in S^{-j,-\infty}, 0\leq j \leq N$ supported in $\Gamma^{\pm}(R^{1/3}, J_2, \s_2)$, $b_j^{\pm} \in S^{-j-k, -\infty}, 0\leq j \leq N$ supported in $\Gamma^{\pm}(R^{1/2}, J_1, \s_1)$ and a operator $R_{N}^{\pm}(t,h)$ such that for $0<h \leq 1$, $\pm t \geq 0$,
\[
e^{-i\frac{t}{h}h^2P}\chi_{\pm} = J_{S_{\pm, R^{1/4}}}(\sum_{j=0}^Nh^ja_j^{\pm})e^{-i\frac{t}{h}h^2P_0}J_{S_{\pm, R^{1/4}}}(\sum_{j=0}^Nh^jb_j^{\pm})^* + h^{N+1}R_{N}(t,h)
\]
where 
\[
J_{S_{\pm, R^{1/4}}}(q)u(x) \coloneqq \frac{1}{(2\pi h)^{n}}\int e^{\frac{i}{h}S_{\pm, R^{1/4}}(x,\x)}q(x,\x)\hat{u}(\x)d\x.
\]
Moreover, for any $s\geq0$, there exists $C_{N,s}>0$ such that for $0<h\leq 1$ and $0\leq \pm t \leq 1/h $, 
\[
\|R_N(t,h)\|_{H^{-s}\to H^s} \leq C_{N,s} h^{N-2s-2}
\]
holds.
\end{prop}

\begin{proof}
The construction of approximate propagators is same as in \cite{BT1}. For the estimates of the remainder terms, we employ the energy estimates Lemma \ref{energy estimates}.

\end{proof}


\subsection{Localized in frequency Strichartz estimates outside a compact set}

\begin{thm}\label{outside}
Let $\f \in C_c^{\infty}(\re\setminus 0)$. Then there exists $R>0$ such that for $\chi \in C_c^{\infty}(\re^n\setminus 0)$    with $\chi =1$ on $|x|<R$ and for any admissible pair $(p.q)$, there exists $C>0$ such that
\[
\|\f(Q)(1-\chi)e^{itP}u_0\|_{L^p([-T,T], L^q)} \leq C\|u_0\|_{L^2}
\]
for $0<h\leq 1$ and $u_0 \in L^2(\re^n)$. Moreover, 
\[
\|\int_0^t \f(Q)(1-\chi) e^{i(t-s)P}(1-\chi)\f(Q) f(s)ds\|_{L^pL^q} \leq C\|f\|_{L^{\frac{p'}{p'-1}}L^{\frac{q'}{q'-1}}}
\]
holds for admissible pair $(p',q')$ with $(p,p') \neq (2,2)$.
\end{thm}

\begin{proof}
For any $N\in \mathbb{N}$, we can write 
\[
\f(Q)(1-\chi) = \sum_{k=0}^N h^k \chi_{k}(x,hD)^{*} + h^{N+1}\tilde{R}_{N}(h)
\]
where $\chi \in C^{\infty}(\re^{2n})$ with $\supp \chi_k \subset \{(x,\x)\in \re^{2n} | |x|\geq R, |\x|^2 \in J\}$ and $\tilde{R}_{N} \in \Psi^{0, -\infty} $. Take $0<\s <1$ satisfying Theorem \ref{eikonal} and $\r_+, \r_- \in C^{\infty}(\re)$ with $\r_+ + \r = 1$, $\r_+(s) = 0$ on $s\leq -\s$ and $\r_-(s) = 0$ on $s\geq \s$. Set $\chi_{k,+}(x,\x) = \chi_k(x,\x)\r_+(g_0(x,\x))$ and $\chi_{k,-}(x,\x) = \chi_k(x,\x)\r_-(g_0(x,\x))$.

Notice that 
\[
\|h^{N+1}\tilde{R}_{N}(h)e^{itP}u_0\|_{L^p([0,1],L^q)} \leq C\|u_0\|_{L^2}
\]
by the Sobolev embedding. By the Keel-Tao theorem (\cite{KT}), it suffices to show that
\[
\|\chi_{k}^{\pm}(x,hD)^*e^{i\frac{t-s}{h}h^2P}\chi_k^{\pm}(x,hD)\|_{L^{\infty}} \leq \frac{C}{(h|t-s|)^{n/2}}\|f\|_{L^1},
\]
for $-T/h\leq t,s \leq T/h$. Notice that
 \begin{align*}
 (\chi_{k}^{\pm}(x,hD)^*e^{i\frac{t-s}{h}h^2P}\chi_k^{\pm}(x,hD))^* = \chi_{k}^{\pm}(x,hD)^*e^{i\frac{s-t}{h}h^2P}\chi_k^{\pm}(x,hD)
\end{align*}
and $\chi_k^{\pm}(x,hD)^*$ is $L^{\infty}$ bounded uniformly in $h$, then it suffices to prove
\[
|K_{\pm}(t-s,x,y,h)| \leq \frac{C}{(h|t-s|)^{n/2}}
\]
for $0<h\leq 1$, $x,y \in \re^n$ and $-T/h\leq s<t \leq T/h$ where $K_{\pm}(t,x,y,h)$ denotes the distribution kernel for $e^{ithP}\chi_{k}^{\pm}(x,hD)$. 

\begin{lem}
There exist $R>0$ and $C>0$, such that
\[
|K_{\pm}(t,x,y,h)| \leq \frac{C}{(\pm t h)^{n/2} }
\]
for $0<h\leq 1$, $0\leq \pm t \leq T/h$ and $x,y \in \re^n$.
\end{lem}
\begin{proof}
Take $a_{\pm}$ and $b_{\pm}$ as Proposition \ref{pr:parx} for $\chi_k^{\pm}$. Then, it suffice to prove that
\[
\frac{1}{(2\pi h)^n}| \int e^{\frac{i}{h}\Phi_{\pm}(t,x,y,\x,R)}a_{\pm}(x,\x)\overline{b_{\pm}(y,\x)} d\x | \leq \frac{C}{(\pm th)^{n/2}}
\]
where $\Phi_{\pm}(t,x,y,\x,R) = S_{\pm}(x,\x)-S_{\pm}(y,\x)-t g_0(\x,\x)$. If $\pm t\leq h$, then the claim is trivial. Thus, we may assume $\pm t\geq h$. By the theorem 3.1, 
\begin{align*}
\frac{\Phi_{\pm}(t,x,y,\x,R)}{t} =& \frac{S_{\pm}(x,\x)-S_{\pm}(y,\x)}{t} -g_0(\x,\x)\\
=& \int_0^1(x-y)\cdot \nabla_{x}S_{\pm}(y+r(x-y),\x)\frac{dr}{t} - g_0(\x,\x)\\
=& \frac{x-y}{t} -g_0(\x,\x) + O(R^{-\m})\frac{x-y}{t}.
\end{align*}
There exist $R_0>0$ and $C_0>0$ such that if $R\geq R_0$ and $|\frac{x-y}{t}| \geq C_0$, we have
\begin{align}
|\frac{\nabla_{\x}\Phi_{\pm}(t,x,y,\x,R)}{t}| \geq& C|\frac{x-y}{t}| \nonumber\\
\label{sta}
|\frac{\pa_{\x}^{\a}\Phi_{\pm}(t,x,y,h,R)}{t}| \leq& C|\frac{x-y}{t}|,\quad |\a| \geq 1
\end{align}
for $|\x|^2 \in J$. Then, non-stationary methods can be applied. 

Thus, we can suppose $|(x-y)/t| \leq C_0$ and $\pm t\geq h$. In this case, we can write
\begin{align*}
\frac{\nabla^2_{\x}\Phi_{\pm}(t,x,y,\x,R)}{t}=-2I_k+O(R^{-\m}).
\end{align*}
Since $(\ref{sta})$ and $|(x-y)/t| \leq C_0$, we can apply the stationary method and get the desired result.
\end{proof}

For inhomogeneous case with $(p,p')\neq (2,2)$, we need the Christ Kiselev lemma.

\begin{lem}[Christ Kislev lemma]\label{CK}
Let $T>0$, $X_1, X_2$ be Banach spaces and $1\leq p <q \leq \infty$. Suppose $K(t,s) \in L^1_{loc}(\re^2, B(X_1,X_2))$ satisfies
\[
\|A\|_{L^p([-T,T], X_1)\to L^q([-T,T], X_2)}  < \infty
\]
where $Af(t) = \int_0^TK(t,s)f(s)ds$. Then, $\tilde{A}f(t)= \int_0^t K(t,s)f(s)ds$ satisfies
\[
\|\tilde{A}\|_{L^p([-T,T], X_1)\to L^q([-T,T], X_2)}  < \infty.
\]
\end{lem}
We refer to \cite{CK} for a proof of this lemma.

Set $v(t) =\int_{-T}^T \f(Q)(1-\chi)e^{i(t-s)P}(1-\chi)\f(Q)f(s)ds$. Then, by the first result, 
\[
\|v\|_{L^pL^q} \lesssim \|\int_{-T}^T e^{-isP}(1-\chi)\f(Q)f(s)ds\|_{L^2}
\]
holds. By the dual of the first result, we conclude 
\[
\|\int_{-T}^T e^{-isP}(1-\chi)\f(Q)f(s)ds\|_{L^2} \lesssim \|f\|_{L^{\frac{p'}{p'-1}}L^{\frac{q'}{q'-1}}}. 
\]
As a result of applying Lemma \ref{CK} with $p \neq \frac{p'}{p'-1}$, we conclude the second result.

\end{proof}

\subsection{Full Strichartz estimates outside a compact set}

\begin{thm}\label{full outside}
There exists $R>0$ such that for any $T>0$, admissible pair $(p,q)$, and $\chi \in C_c^{\infty}(\re^n)$ with $\chi =1$ on $|x| <R$, there exists $C>0$ such that
\[
\|(1-\chi)e^{itP}u_0\|_{L^p([-T,T], L^q(\re^n))} \leq C\|u_0\|_{L^2(\re^n)}
\]
for $u_0 \in L^2(\re^n)$.
\end{thm}

\begin{proof}
By Corollary \ref{PLsum} , 
\[
\|(1-\chi)e^{itP}u_0\|_{L^pL^q} \lesssim \|u_0\|_{L^2} + (\sum_{h:dyadic}\|\f(Q)(1-\chi)e^{itP}u_0\|_{L^pL^q})
\] 
holds. Then, it suffices to estimates $\|\f(Q)(1-\chi)e^{itP}u_0\|_{L^pL^q}$.
Take $\g\in C_c^{\infty}(\re^n)$ with $\g=1$ on $\supp \f$ and set $v(t) = (1-\chi)\g(Q)e^{itP}u_0$.  We note $v(t)$ satisfies
\begin{align*}
v(t) = (1-\chi)e^{itP}\g(Q)u_0 -i(1-\chi)\int_0^t e^{i(t-s)P}[P, \g(Q)]e^{isP}u_0ds.
\end{align*}
By Theorem \ref{outside},  we compute
\begin{align*}
\|&\f(Q)(1-\chi)\int_0^T e^{i(t-s)P}[P, \g(Q)]e^{isP}u_0ds\|_{L^pL^q}\\
&\lesssim \int_0^T \|e^{-isP}[P, \g(Q)]e^{isP}u_0\|_{L^2}ds\\
&= \|e^{-isP}[P, \g(Q)]e^{isP}u_0\|_{L^1L^2}.
\end{align*}
By this computation and the Christ-Kislev lemma (remark that $p\geq 2>1$),  we have
\begin{align*}
\|\f(Q)v\|_{L^pL^q}&\lesssim \|\g(Q)u_0\|_{L^2} + \|e^{-isP}[P, \g(Q)]e^{isP}u_0\|_{L^1L^2}\\
&\lesssim  \|\g(Q)u_0\|_{L^2}+h\|u_0\|_{L^2}
\end{align*}
Now we notice
\[
\f(Q)(1-\chi)e^{itP}u_0= \f(Q)v(t)+ O(h^{\infty})e^{itP}u_0,
\]
and we conclude
\[
\|\f(Q)(1-\chi)e^{itP}u_0\|_{L^pL^q} \lesssim \|\g(Q)u_0\|_{L^2} + h\|u_0\|_{L^2}.
\]
\end{proof}


\section{Estimates for relatively compact region}
In this section, we prove the Strrichartz estimates on any compact set by using the local smoothing effects. 
 Due to the non-commutativity  between $\f(Q)$ and $e^{itP}$, we need more delicate calculations than the elliptic cases. By the time reversibility, we may prove the Strichartz estimates when we replace the time interval $[-T,T]$ with $[0,1].$
\begin{prop}\label{WKB}
Let $a(x,\x,h) \in S^0$ with a decomposition $a_0+a_1$, where $a_0(x,\x,h)\in S^0$ with 
\[
\supp a_0(\cdot,\cdot,h) \subset \re^n \times B_R(0)
\]
for some $h$-independent constant $R>0$ and $a_1(x,\x,h) \in h^mS^{-l}$ for some $m, l >n/2$.
Then, there exist $\a>0$ for any $N\in \mathbb{N}$, there exist operators $J_N(t,h)$, $R_N(t,h)$ and a constant $C_N>0$ for any $h\in (0,1]$,  we can write $e^{ithP}a(x,hD,h) = J_N(t,h) + R_{N,1}(t,h) + R_2(t,h)$ for $t\in [-\a,\a]$ and 
\begin{align*}
\|J_N(t,h)\|_{L^1(\re^n)\to L^{\infty}(\re^n)} \leq& \frac{C_N}{(|t|h)^{\frac{n}{2}}},\\
\|R_{N,1}(t,h)\|_{L^2(\re^n)\to L^2(\re^n)} \leq& C_Nh^{N+1}\\
\|b(x,hD,h)R_2(t,h)\|_{L^1(\re^n)\to L^{\infty}(\re^n)} \leq& C_Nh^{m-n}
\end{align*}
for $b\in S^l$ with $l< -n$.
\end{prop}

\begin{proof}
Set $R_2(t,h) = e^{ithP}a_1(x,hD)$. By virtue of the $L^2\to L^{\infty}$ estimates for pseudodifferential operators and the self-adjointness of $P$, we have
\begin{align*}
\|b(x,hD)R_2(t,h)\|_{L^1(\re^n)\to L^{\infty}(\re^n)} &\leq Ch^{-\frac{n}{2}} \|R_2(t,h)\|_{L^1(\re^n)\to L^2(\re^n)} \\
&\leq Ch^{m-n}.
\end{align*}
Therefore, it is enough to find $J_N$ and $R_N$ as this statement of the proposition satisfying
\[
e^{ithP}a_0(x,hD,h) = J_N(t,h) + R_N(t,h).
\]
The proof is the standard WKB approximation argument and given in \cite{S}.
\end{proof}

\begin{cor}\label{WKBcor}
Let $\f,\g \in C_c^{\infty}(\re)$ with $\g=1$ on $\supp \f$. Then, there exist $\a>0$ and $C>0$ such that for every $J \subset \re$ with $|J|\leq \a h$ and $h\in (0,1]$, admissible pair $(p,q)$, if $u$ saitsfies
\[
i\pa_t u + Pu = 0,\quad u|_{t=0} =u_0, \quad u_0\in L^2(\re^n)
\] 
then
\[
\|\g(Q)u\|_{L^p(J,L^q(\re^n))} \leq C\|u_0\|_{L^2(\re^n)}.
\]
Moreover, if $u$ satisfies
\[
i\pa_tu + Pu = \f(Q)f, \quad u|_{t=0} = 0
\]
then
\[
\|\g(Q)u\|_{L^p(J,L^q(\re^n))} \leq C\|\f(Q)f\|_{L^{p_1}(J,L^{q_1}(\re^n))}
\]
holds.
\end{cor}

\begin{proof}
We can write $\f(Q) = a(x,hD,h)$, then $a$ satisfies the condition for Proposition \ref{WKB}. Applying the Keel-Tao theorem, we get the desired estimates.

\end{proof}

\begin{prop}
Let $\f,\g \in C_c^{\infty}(\re)$ and $\chi \in C_c^{\infty}(\re^n)$. Then for any $N\in \mathbb{N}_{\geq 0}$ there exists $C>0$ such that $v= \f(Q)\chi e^{itP}u_0$ satisfies 
\begin{align*}
\|v\|_{L^p([0,1],L^q(\re^n))} \leq& C\|v\|_{L^{\infty}([0,1],L^2(\re^n))} +Ch^{-\frac{1}{2}}\|v\|_{L^2([0,1],L^2(\re^n))}\\
&+ Ch^{\frac{1}{2}}\|[P,\f(Q)]\chi e^{itP}u_0 + \f(Q)[P,\chi]e^{itP}u_0 \|_{L^2([0,1],L^2(\re^n))}\\
&+ h^{N} \|u_0\|_{L^2}
\end{align*}
for $h\in (0,1]$, $u_0\in L^2(\re^n)$.
\end{prop}

\begin{proof}
$v$ satisfies
\[
i\pa_tv+Pv = [P,\f(Q)]\chi e^{itP}u_0 + \f(Q)[P,\chi]e^{itP} u_0, \quad v|_{t=0} = \f(Q)\chi u_0.
\]

We can decompose 
\[
[0,1] \subset \cup_{j=0}^{k} J_j
\]
where $J_j =[j\a h,(j+1)\a h]$ and $k\lesssim \frac{1}{h}$. By Corollary \ref{WKBcor}, 
\begin{align*}
\|v\|_{L^p(J_0,L^q)} =& \|\g(Q)v\|_{L^pL^q}\\
\lesssim& \|v\|_{L^{\infty}(J_0,L^2(\re^n))}+ h^N\|u_0\|_{L^2}\\
&+ h^{\frac{1}{2}}\|[P,\f(Q)]\chi e^{itP}u_0 + \f(Q)[P,\chi]e^{itP} u_0\|_{L^2(J_0,L^2(\re^n))}\\
\|v\|_{L^p(J_k,L^q(\re^n))} \lesssim& \|v\|_{L^{\infty}(J_k,L^2(\re^n))}+ h^N\|u_0\|_{L^2}\\
&+ h^{\frac{1}{2}}\|[P,\f(Q)]\chi e^{itP}u_0 + \f(Q)[P,\chi]e^{itP} u_0\|_{L^2(J_k,L^2(\re^n))}.
\end{align*}
Take $\y \in C_c^{\infty}(\re)$ such that $\y = 1$ on $[-\frac{1}{2},\frac{1}{2}]$, $\y = 0$ outside $[-\frac{3}{4},\frac{3}{4}]$. Set $\y_j(t) = \y(\frac{t-(j+\frac{1}{2})}{\a h})$ for $j=1,... k-1$,
\[ 
w(t) = v(t) - i\int_0^te^{i(t-s)P}(1-\g(Q))[P,\f(Q)]\chi e^{isP}u_0ds,
\]
$v_j(t) = \y_j(t)v(t)$ and $w_j(t) = \y_j(t)w(t)$. Then, $v_j-w_j$ solves
\begin{align*}
&i\pa_t(v_j-w_j) + P(v_j-w_j)\\
&= i\y_j'(t)(v-w) + \y_j(t)\g(Q)([P,\f(Q)]\chi e^{itP}u_0 + \f(Q)[P,\chi]e^{itP} u_0),\\
&v_j(0)-w_j(0) = 0.
\end{align*}
Set $J_j' = J_j + [-\frac{\a h}{2},\frac{\a h}{2}]$, then we have
\begin{align*}
\|\g&(Q)(v-w)\|_{L^p(J_j,L^q(\re^n))}\\
\leq& \|\g(Q)(v_j-w_j)\|_{L^p(J_j',L^q(\re^n))}\\
\lesssim & \frac{1}{h}\|\g(Q) \int_0^te^{i(t-s)P}\y'(\frac{s-(j+\frac{1}{2})}{\a h})(v-w)ds\|_{L^p(J_j',L^q(\re^n))} \\
&+ \|[P,\f(Q)]\chi e^{itP}u_0 + \f(Q)[P,\chi]e^{itP} u_0\|_{L^1(J_j',L^2(\re^n))}\\
\leq& \frac{1}{h}\|v\|_{L^1(J_j',L^2(\re^n))} + h^N\|u_0\|_{L^2}\\
&+ \|[P,\f(Q)]\chi e^{itP}u_0 + \f(Q)[P,\chi]e^{itP} u_0\|_{L^1(J_j',L^2(\re^n))}\\
\leq& \frac{1}{h^{1/2}}\|v\|_{L^2(J_j',L^2(\re^n))} + h^N\|u_0\|_{L^2}\\
&+  h^{1/2} \|[P,\f(Q)]\chi e^{itP}u_0 + \f(Q)[P,\chi]e^{itP} u_0)\|_{L^2(J_j',L^2(\re^n)},
\end{align*}
where we use Christ-Kislev lemma as in the proof of Theorem \ref{full outside}.
By using

\hspace{-7mm} $\|w\|_{L^p(J_j',L^q(\re^n))}  \lesssim \|v\|_{L^p(J_j',L^q(\re^n))} +O(h^{\infty})\|u_0\|_{L^2}$, we have
\begin{align*}
\|v\|_{L^p(J_j,L^q(\re^n))} \lesssim& \frac{1}{h^{1/2}}\|v\|_{L^2(J_j',L^2(\re^n))} + h^N\|u_0\|_{L^2}\\
&+  h^{1/2} \|[P,\f(Q)]\chi e^{itP}u_0 + \f(Q)[P,\chi]e^{itP} u_0\|_{L^2(J_j',L^2(\re^n))}.
\end{align*}
Consequently, we get
\begin{align*}
\|v\|_{L^p(J_j,L^q)}^p \lesssim& \|v\|_{L^{\infty}(J_j',L^2)}^p+ \sum_{j=0}^k\frac{1}{h^{p/2}}\|v\|_{L^2(J_j',L^2(\re^n))}^p + h^{pN-1}\|u_0\|_{L^2}^p\\
&+  \sum_{j=0}^k h^{p/2} \|[P,\f(Q)]\chi e^{itP}u_0 + \f(Q)[P,\chi]e^{itP} u_0\|_{L^2(J_j',L^2)}^p\\
\lesssim& \|v\|_{L^{\infty}(J_j',(L^2)}^p+ \frac{1}{h^{p/2}}(\sum_{j=0}^k\|v\|_{L^2(J_j',L^2)}^2)^{p/2} + h^{pN-1}\|u_0\|_{L^2}^p\\
&+   h^{p/2} (\sum_{j=0}^k\|[P,\f(Q)]\chi e^{itP}u_0 + \f(Q)[P,\chi]e^{itP} u_0\|_{L^2(J_j',L^2)}^2)^{p/2}.
\end{align*}

\end{proof}

\begin{rem}
Since $[\f(Q),P] \neq 0$, we cannot use the method Proposition 5.4 in $\cite{BT1}$ directly. So we modified their argument by using the estimate $[\f(Q),P]=O(h)$ and the energy estimates.

\end{rem}

\begin{thm}
Under Assumption A,B and C, for any $\chi \in C_c^{\infty}(\re^n)$ there exists $C>0$ such that
\begin{align*}
\|\chi e^{itP}u_0\|_{L^p([0,1],L^q(\re^n))}\leq \|u_0\|_{L^2(\re^n)}
\end{align*}
for $u_0\in L^2(\re^n)$.
\end{thm}

\begin{proof}
 Due to the last proposition,
\begin{align*}
\|v\|_{L^p([0,1], L^p(\re^n))} \lesssim& \|v\|_{L^{\infty}([0,1], L^2(\re^n)))} + h^{-\frac{1}{2}}\|v\|_{L^2([0,1], L^2(\re^n)))} \\
&+ h^{\frac{1}{2}}\|[P,\f(Q)]\chi e^{itP}u_0 + \f(Q)[P,\chi ]e^{itP}u_0\|_{L^2([0,1], L^2(\re^n)))}\\
&+ h^{N-\frac{1}{2}-1}\|u_0\|_{L^2(\re^n)}\\
\end{align*}
holds.
\begin{align*}
\|v\|_{L^{\infty}([0,1], L^2(\re^n)))}  \lesssim& \|\chi e^{itP} \f(Q)u_0\|_{L^{\infty}([0,1],L^2(\re^n))}\\
&+ \|[\f(Q), \chi e^{itP}]u_0\|_{L^{\infty}([0,1],L^2(\re^n))}\\
\lesssim& \|\f(Q)u_0\|_{L^2(\re^n)} + h\|u_0\|_{L^2(\re^n)}.
\end{align*}

\begin{align*}
h^{-\frac{1}{2}}\|v\|_{L^2([0,1], L^2(\re^n)))} \lesssim& h^{-\frac{1}{2}}\|\f(Q)[\g(Q) , \chi e^{itP}]u_0\|_{L^2([0,1], L^2)}\\
&+ h^{-\frac{1}{2}}\|\f(Q)\chi e^{itP}\g(Q)u_0\|_{L^2([0,1], L^2)}\\
\lesssim& h^{\frac{1}{2}}\|u_0\|_{L^2} + \|\g(Q)u_0\|_{L^2}
\end{align*}
by the local smoothing estimates. At last, by using the local smoothing estimates again, we get

\begin{align*}
h^{\frac{1}{2}}\|\f(Q)[P,\chi ]e^{itP}u_0\|_{L^2([0,1], L^2(\re^n)))} \lesssim& h^{\frac{1}{2}}\|\f(Q)[\g(Q), [\chi, P]e^{itP}]u_0\|_{L^2([0,1], L^2)}\\
&+ h^{\frac{1}{2}}\|\f(Q)[\chi, P]e^{itP}\g(Q)u_0\|_{L^2([0,1], L^2)}\\
\lesssim& h^{\frac{1}{2}}\|u_0\|_{L^2} + \|\g(Q)u_0\|_{L^2}.
\end{align*}

Consequently, we have
\begin{align*}
\|v\|_{L^p([0,1], L^q)} \lesssim \|\g(Q)u_0\|_{L^2} + h^{\frac{1}{2}}\|u_0\|_{L^2} 
\end{align*}
and by Corollary \ref{PLsum},  we conclude
\[
\|\chi e^{itP}u_0\|_{L^p([0,1],L^q)} \lesssim \|u_0\|_{L^2}
\]
\end{proof}


\bibliographystyle{jplain}
\bibliography{reference}

\end{document}